\documentclass[12pt]{amsart}

\usepackage{amsfonts,amssymb,amsmath,textcomp}

 2

\newtheorem{thm}{Theorem}[section]
\newtheorem{lem}[thm]{Lemma}
\newtheorem{prop}[thm]{Proposition}

\newcommand{\thmref}[1]{Theorem~\ref{#1}}
\newcommand{\lemref}[1]{Lemma~\ref{#1}}

\newcommand{\propref}[1]{Proposition~\ref{#1}}

\newtheorem{rmk}[thm]{Remark}
\begin{document}
\title[Oscillations of Fourier coefficients of automorphic forms]
{Oscillations of coefficients of Dirichlet series attached to automorphic forms}
\author{Jaban Meher and M. Ram Murty}

\address[Jaban Meher]{School of Mathematical Sciences,
                                       National Institute of Science Education and Research,
                                        Bhubaneswar, Via-Jatni, Khurda 752050, Odisha, India.}
\email{jaban@niser.ac.in}

\address[M. Ram Murty]{Department of Mathematics,
                                            Queen's University, Kingston,
                                            Ontario, K7L 3N6, Canada.}
\email{murty@mast.queensu.ca}


\date{\today}


\maketitle 
\begin{abstract}
For $m\ge 2$, let $\pi$ be an irreducible cuspidal automorphic representation of 
$GL_m(\mathbb{A}_{\mathbb{Q}})$ with unitary central character. Let $a_\pi(n)$ be the
$n^{th}$ coefficient of the $L$-function attached to $\pi$. 
Goldfeld and Sengupta have recently obtained  a bound for 
$\sum_{n\le x} a_\pi(n)$ as $x \rightarrow \infty$. For $m\ge 3$ and $\pi$ not a symmetric power
of a $GL_2(\mathbb{A}_{\mathbb{Q}})$-cuspidal automorphic representation with not all finite primes unramified for $\pi$, their bound is better than all previous bounds. 
In this paper, we further improve the bound of Goldfeld and Sengupta. We also prove a quantitative result for the number of sign changes of the coefficients of certain automorphic $L$-functions, provided the coefficients are real numbers.
\end{abstract}

\section{Introduction}
In an earlier paper \cite{MM}, we described a general theorem to study the sign changes of any 
sequence of real numbers ${\{a(n)\}}_{n=1}^{\infty}$. More precisely, we proved the following.
\begin{thm}\label{thm1}
Let ${\{a(n)\}}_{n=1}^{\infty}$ be a sequence of real numbers such that
\begin{enumerate}
\item[(i)]
$a(n)=O(n^\alpha)$,
\item[(ii)]
$\sum_{n\le x}a(n)=O(n^\beta)$,
\item[(iii)]
$\sum_{n\le x}a(n)^2=cx+O(x^\gamma)$,
\end{enumerate}

with $\alpha, \beta, \gamma,c\ge 0$. If $\alpha +\beta<1$, then for any $r$ satisfying
$$
\rm{max}\{\alpha+\beta, \gamma\}< \mbox{r}<1,
$$
the sequence ${\{a(n)\}}_{n=1}^{\infty}$ has at least one sign change for $n\in [x,x+x^r]$.
Consequently, the number of sign changes of $a(n)$ for $n\le x$ is $\gg x^{1-r}$ for sufficiently 
large $x$.
\end{thm}
We applied this theorem to study oscillations of Fourier coefficients of modular forms of 
half-integral weight. 
In this paper, we focus our attention on coefficients of Dirichlet series 
attached to automorphic forms.

Let $F$ be an algebraic number field with $[F:\mathbb{Q}]=d$, and ring of integers $\mathcal{O}_F$
and discriminant $D_F$. For $n\ge 2$, let $\pi$ be a cuspidal automorphic form on
$GL_n(\mathbb{A}_F)$ with unitary central character and where $\mathbb{A}_F$ denotes the 
adele ring of $F$. We denote by $\tilde{\pi}$ the contragradient representation of $\pi$. The standard
$L$-function attached to $\pi$ for $\Re(s)>1$ is given by
$$
L(s,\pi):=\prod_{v}L(s,\pi_v),
$$
where the product is over all the places of $F$ (finite and infinite) and $L(s,\pi_v)$ is the local
Euler factor described as follows. At an unramified finite place $v$ of $\pi$,
$$
L(s,\pi_v)=\prod_{i=1}^{n}(1-\alpha_{i,\pi}(v)(Nv)^{-s})^{-1},
$$
for suitable complex numbers $\alpha_{i,\pi}(v)$ (called the Satake parameters) and $Nv$ is the
absolute norm of $v$. At an unramified infinite place $v$,
$$
L(s,\pi_v)=\prod_{i=1}^{n}\Gamma_v(s-\mu_{i,\pi}(v)),
$$
where the $\mu_{i, \pi}(v)$'s are complex numbers and
$$\Gamma_v(s)=\pi^{-s/2}\Gamma(s/2) ~~~~~\mbox{if}~ v ~\mbox{is~ real}$$
and 
$$\Gamma_v(s)=(2\pi)^{-s}\Gamma(s) ~~~~~\mbox{if} ~v~\mbox{is~ complex}.$$
With this notation, the generalized Ramanujan conjecture is the assertion that
$$|\alpha_{i,\pi}(v)|=1 ~~~~~\mbox{if}~v~ \mbox{is~ finite}$$
and
$$\Re(\mu_{i,\pi}(v))= 0 ~~~~~\mbox{if}~v~\mbox{is~infinite.}$$
In this context, Luo, Rudnick and Sarnak \cite{LRS} have proved that
$$|\alpha_{i,\pi}(v)|\le (Nv)^{\frac{1}{2}-\frac{1}{n^2+1}} ~~~~~\mbox{if}~v~ \mbox{is~ finite}$$
and
$$|\Re(\mu_{i,\pi}(v))|\le \frac{1}{2}-\frac{1}{n^2+1} ~~~~~\mbox{if}~v~\mbox{is~infinite.}$$
This (at present) is the most general result. Better results are there for $n=2$.

Godement and Jacquet \cite{GJ} indicated how to define $L(s,\pi_v)$ for the finitely many $v$ 
at which $\pi_v$ is ramified and the completed $L$-function admits an analytic continuation to 
the entire complex plane and satisfies a functional equation relating $L(s,\pi)$ to $L(1-s,\tilde{\pi})$.
We refer the reader to \cite{BG} (Chapter $9$ in particular) for the details regarding the analytic
properties of these $L$-series.

Jacquet, Piatetski-Shapiro and Shalika \cite{JPS} developed the theory of Rankin-Selberg for 
general $L$-series. More precisely, if $\pi_1$ is a cuspidal automorphic representation on 
$GL_n(\mathbb{A}_F)$ and $\pi_2$ a cuspidal automorphic representation of 
$GL_m(\mathbb{A}_F)$ then one can construct $L(s,\pi_1\times \pi_2)$ using the
local Satake parameters of $\pi_1$ and $\pi_2$. It turns out that $L(s,\pi_1\times \pi_2)$ extends
to an entire function if $m\ne n$. Otherwise, $L(s,\pi_1\times \pi_2)$ extends to an analytic function
for all $s\in \mathbb{C}$ except for $s=i\sigma$ and $s=1+i\sigma$ with $\sigma\in \mathbb{R}$,
where $\widetilde{\pi}_1=\pi_2\times|\rm{det}|^{i\sigma}$. In this article we will have $\mathbb{F}=\mathbb{Q}$.
For $m\ge 2$, let $\pi$ be an irreducible cuspidal automorphic representation of 
$GL_m(\mathbb{A}_{\mathbb{Q}})$ with unitary central character. For $\Re(s)>1$, the $L$-function
attached to $\pi$ is given by
$$
L(s,\pi)=\prod_{p}L(s,\pi_p),
$$
where $p$ runs over all finite rational primes and $L(s,\pi_p)$ is the local Euler factor. 
For $\Re(s)>1$, let the Dirichlet series 
representation of $L(s,\pi)$ be
$$
L(s,\pi)=\sum_{n=1}^{\infty}\frac{a_{\pi}(n)}{n^s}.
$$
Then the completed $L$-function
$$
\Lambda(s,\pi):=\prod_{i=1}^{m}\pi^{-\frac{s+\lambda_i}{2}}
\Gamma \left(\frac{s+\lambda_i}{2}\right)L(s, \pi),
$$
where $\lambda_i\in \mathbb{C}$ (for $i=1,2,\dots, m$),
admits an analytic continuation to the entire complex plane and
satisfies the functional equation
$$
\Lambda(s,\pi)=\epsilon_\pi N_\pi^{\frac{1}{2}-s}\Lambda(1-s,\tilde{\pi}),
$$
where $\epsilon_\pi\in \mathbb{C}$ has absolute value $1$ and is called the root number,
$N_\pi$ is a positive integer called the conductor of $\pi$.  
A prime $p$ is called unramified for $\pi$ if $(p,N_\pi)=1$.
\begin{thm}\label{improvement}
For $m\ge 2$, let $\pi$ denote an irreducible cuspidal automorphic representation of 
$GL_m(\mathbb{A}_{\mathbb{Q}})$ with unitary central character. Let $L(s,\pi)$ be its 
associated $L$-function with $n^{th}$ coefficient $a_{\pi}(n)$. Then we have
$$
\sum_{n\le x}a_\pi(n)\ll
\left\{ \begin{array}{rcl}
x^{\frac{71}{192}+\epsilon} & \mbox{if} & m=2,\\
x^{\frac{m^2-m}{m^2+1}+\epsilon} & \mbox{if} & m\ge 3.
\end{array}\right. 
$$
for any fixed $\epsilon>0$.
\end{thm}
The method of proof of the above theorem is same as in \cite{GS}, but we use a result of L\"{u}
\cite{LU} which ultimately follows from the result of Chandrasekharan and Narasimhan \cite{CN}.
The difference in our proof is that we estimate differently which is similar to L\"{u} \cite{LU} and get a 
better bound in the general case too. We apply the above theorem to deduce the following result on 
sign changes of coefficients of $L$-series attached to certain automorphic forms. However, the bounds of Golfeld and Sengupta \cite{GS} for the first moment of the coefficients of the automorphic $L$-functions are enough to prove the following result.

\begin{thm}\label{main}
Let $L(s,\pi)$ be the automorphic $L$-function associated to an automorphic irreducible 
self-dual cuspidal representation $\pi$ of $GL_2(\mathbb{A}_{\mathbb{Q}})$
and let $a_\pi(n)$ be its $n^{th}$ coefficient. If the sequence $\{a_\pi(n)\}_{n=1}^{\infty}$ is a sequence of real numbers, then for any real number $r$ satisfying $\frac{3}{5}<r<1$, the sequence 
$\{a_\pi(n)\}_{n=1}^{\infty}$ has at least one sign change for $n\in [x,x+x^r]$.
Consequently, the number of sign changes of $a_\pi(n)$ for $n\le x$ is $\gg x^{1-r}$ 
for sufficiently large $x$.
\end{thm}
\begin{rmk}
Although the above theorem is stated for automorphic forms for $GL_2$, we will see that we get a similar sign change result for certain automorphic forms for $GL_m$, where $m>2$. For example, we will see that
a similar result is valid for symmetric square $L$-functions attached to Hecke Maass cusp forms. Also we get similar result for certain symmetric power $L$-functions attached to any holomorphic cusp form.
\end{rmk}
Next, we prove a sign change result for the eigenvalues of Siegel Hecke cusp form of genus $2$. To state the result we review some basics of spinor zeta function attached to Siegel modular forms of genus $2$.

Let $F$ be a non-zero Siegel cusp form of weight $k$ for the group $Sp_4(\mathbb{Z})$ 
which is an eigenfunction for all the Hecke operators $T(n)$ with eigenvalues $\lambda_F(n)$. It is 
well known that $\lambda_F(n)$ are real and multiplicative. Assume also that $F$ is not a Saito-Kurokawa lift. We define the normalized eigenvalue of $F$ by

$$
\lambda(n):=\frac{\lambda_F(n)}{n^{k-3/2}}.
$$
The spinor zeta function attached to $F$ is defined by
$$
Z(s,F):=\prod_p Z_p(s,F),
$$
where 
$$
Z_p(s,F)=\prod_{1\le i\le 4}(1-\beta_{i,p}p^{-s})^{-1}.
$$
Here $\beta_{1,p} :=\alpha_{0,p}$, $\beta_{2,p}:=\alpha_{0,p}\alpha_{1,p}$, 
$\beta_{3,p}:=\alpha_{0,p}\alpha_{2,p}$, $\beta_{4,p}:=\alpha_{0,p} \alpha_{1,p}\alpha_{2,p}$, and 
$\alpha_{0,p},\alpha_{1,p}, \alpha_{2,p}$ are the Satake parameters of $F$. The 
Ramanujan conjecture which has been proved by Weissauer \cite{W}, is the assertion that for 
all primes $p$ one has
\begin{equation}\label{ramanujan}
|\alpha_{0,p}|=|\alpha_{1,p}|=|\alpha_{2,p}|=1.
\end{equation}
For $\Re(s)>1$, let the Dirichlet series representation of $Z(s,F)$ be
$$
Z(s,F)=\sum_{n=1}^{\infty}\frac{a_F(n)}{n^s}.
$$
The completed $L$ function
$$
\Lambda(s,F):= (2\pi)^{-s}\Gamma\left(s+k-\frac{3}{2}\right)
\Gamma\left(s+\frac{1}{2}\right)Z(s,F)
$$
can be analytically continued to the whole complex plane and it satisfies the functional equation
$$
\Lambda(s,F)=(-1)^k \Lambda(1-s,F).
$$
It is well known that
\begin{equation}\label{dir}
\sum_{n=1}^{\infty}\frac{\lambda(n)}{n^s}=Z(s,F)\zeta(2s+1)^{-1}.
\end{equation}
From \eqref{dir}, we easily see that
\begin{equation}\label{mobius}
\lambda(n)=\sum_{d^2m=n}\frac{\mu(d)}{d}a_F(m),
\end{equation}
where $\mu$ is the M\"{o}bius function. Using \eqref{ramanujan}, we see that for every integer
$n\ge 1$, we have
$$
|a_F(n)|\le d_4(n),
$$
where $d_l(n)$ is the number of ways of writing $n$ as a product of $l$ positive integers.
Now by using \eqref{mobius} it can be easily deduced that for every $\epsilon >0$, we have
\begin{equation}\label{siegel-ramanujan}
|\lambda(n)|\le d_5(n)\ll n^{\epsilon}.
\end{equation}
Kohnen \cite{K} has proved that the sequence $\{\lambda_F(n)\}_{n\ge 1}$ changes signs infinitely often. We prove the following quantitative result on sign changes of the sequence 
$\{\lambda_F(n)\}_{n\ge 1}$.

\begin{thm}\label{siegel}
If $F$ is a non-zero Siegel cusp form of weight $k$ for the group $Sp_4(\mathbb{Z})$ 
which is an eigenfunction for all the Hecke operators $T(n)$ with eigenvalues $\lambda_F(n)$ 
and which is not a Saito-Kurokawa lift, then for any real number $r$ satisfying 
$\frac{41}{47}<r<1$, the sequence $\{\lambda_F(n)\}_{n= 1}^{\infty}$ has at least one sign change
for $n\in [x, x+x^r]$. Consequently, the number of sign changes of $\{\lambda_F(n)\}$ for $n\le x$
is $\gg x^{1-r}$ for sufficiently large $x$.
\end{thm}

\section{Preparatory results}
We recall a well-known result of Chandrasekharan and Narasimhan \cite{CN}.
We need to define a few things before proceeding to the theorem.

Let $$\phi(s)=\sum_{n\ge 1}\frac{a(n)}{n^s},$$ and  $$\psi(s)=\sum_{n\ge 1}\frac{b(n)}{n^s}$$ be two Dirichlet series.
Let $\Delta(s)=\prod_{i=1}^{l}\Gamma(\alpha_i s+\beta_i)$ and $A=\sum_{i=1}^{l}\alpha_i$. Assume that
$$Q(x)=\frac{1}{2\pi i}\int_{\mathcal{C}}\frac{\phi(s)}{s}x^sds,$$
where $\mathcal{C}$ encloses all the singularities of the integrand. With these definitions, we now state the result \cite[Theorem $4.1$]{CN}.
\begin{thm}\label{chandra-nara}
Suppose that the functional equation
$$
\Delta(s)\phi(s)=\Delta(\delta-s)\psi(\delta-s)
$$
is satisfied with $\delta>0$, and that the only singularities of the function $\phi$ are poles. Then, we have
$$
B(x)-Q(x)=O(x^{\frac{\delta}{2}-\frac{1}{4A}+2A\eta u})
+O(x^{q-\frac{1}{2A}-\eta}(\log x)^{r-1})
+O(\sum_{x<n\le x'}|a(n)|),
$$
for every $\eta\ge 0$, where $B(x)=\sum_{n\le x}a(n)$, $x'=x+O(x^{1-\frac{1}{2A}-\eta})$, $q$ is the maximum
of the real parts of the singularities of $\phi$, $r$ the maximum order of a pole with real part $q$, and
$u=\beta-\frac{\delta}{2}-\frac{1}{4A}$, where $\beta$ is such that $\sum_{n=1}^{\infty}|b(n)|n^{-\beta}$ is finite.
If, in addition, $a_n\ge 0$, then we have
$$
B(x)-Q(x)=O(x^{\frac{\delta}{2}-\frac{1}{4A}+2A\eta u})
+O(x^{q-\frac{1}{2A}-\eta}(\log x)^{r-1}).
$$ 
\end{thm}

We recall some basic results about the Rankin-Selberg $L$-function $L(s,\pi \times \tilde{\pi})$
associated to $\pi $ and its contragredient $\tilde{\pi}$. We write $L(s,\pi \times \tilde{\pi})$ 
as a Dirichlet series
\begin{equation}\label{rankin-selberg-dirichlet}
L(s,\pi \times \tilde{\pi})=\sum_{n=1}^{\infty}\frac{a_{\pi \times \tilde{\pi}}(n)}{n^s}.
\end{equation}
Then the completed $L$-function
\begin{equation}\label{completed1}
\Lambda(s, \pi \times \tilde{\pi})
:= \prod_{i=1}^{r} \prod_{j=1}^{r} \pi^{-s-\frac{-\lambda_i-\lambda_j}{2}}
\Gamma\left(\frac{s+\lambda_i}{2}\right)\Gamma\left(\frac{s-\lambda_j}{2}\right)
L(s, \pi \times \tilde{\pi})
\end{equation}
has a meromorphic continuation to the entire complex plane and
satisfies the functional equation
\begin{equation}\label{functional1}
\Lambda(s, \pi \times \tilde{\pi})=\epsilon_{\pi \times \tilde{\pi}}N_{\pi \times \tilde{\pi}}^{\frac{1}{2}-s}
\Lambda(1-s, \pi \times \tilde{\pi}).
\end{equation}
It is also known that $a_{\pi \times \tilde{\pi}}\ge 0$ and the $L$-function $L(s, \pi \times \tilde{\pi})$ 
has a simple pole at $s=1$. For $(n,N_\pi)=1$, it has been shown in \cite{GS} that
\begin{equation}\label{ineq}
|a_\pi(n)|^2\leq a_{\pi \times \tilde{\pi}}(n).
\end{equation}
Then by equation (3.10) of \cite{LU}, we have the following result.
\begin{prop}\label{prop1}
We have
$$
\sum_{n\le x}a_{\pi \times \tilde{\pi}}(n)=c_\pi x +O_{\epsilon,\pi}(x^{\frac{m^2-1}{m^2+1}+\epsilon}),
$$
where $c_\pi$ is a positive constant.
\end{prop}
We also state the following result 
\begin{thm}\label{gun-murty}
Let $L(s,\pi_1)$ and $L(s,\pi_2)$ be two automorphic $L$-functions whose associated 
Dirichlet series can be written as
$$
\sum_{n=1}^{\infty}\frac{a(n)}{n^s}~~~~~~~~~~~\mbox{and}~~~~~~~~~~~
\sum_{n=1}^{\infty}\frac{b(n)}{n^s}
$$
respectively. Suppose that the exponents towards the Ramanujan's conjecture for $\pi_1$ and $\pi_2$ are strictly less than $1/4$. Then the series
$$
\sum_{n=1}^{\infty}\frac{a(n)b(n)}{n^s}=L(s,\pi_1\times \pi_2)g(s),
$$
where $g(s)$ is a Dirichlet series absolutely convergent for $\Re(s)>1/2$ and 
$L(s,\pi_1\times \pi_2)$ is the Rankin-Selberg $L$-series attached to $\pi_1$ and $\pi_2$.
\end{thm} 
The above result is proved in \cite[Theorem 6]{GM} where it is assumed that $\pi_1$ and $\pi_2$ 
satisfy Ramanujan's conjecture. But the proof goes through with a weaker hypothesis that the exponents towards the Ramanujan's conjecture for $\pi_1$ and $\pi_2$ are strictly less than $1/4$. Next we prove the non-vanishing of the Dirichlet series $g(s)$ appearing in the above theorem at the point $1$ in our next proposition.
\begin{prop}\label{proposition}
Let $\pi_1=\tilde\pi_2=\pi$ be self-dual in \thmref{gun-murty}. Then
$$
\sum_{n=1}^{\infty}\frac{{|a_\pi(n)|}^2}{n^s}=L(s,\pi \times \pi)g(s).
$$
If $g(s)=\displaystyle\sum_{n=1}^{\infty}\frac{c(n)}{n^s}$, then the series 
$g(1)=\displaystyle\sum_{n=1}^{\infty}\frac{c(n)}{n}$ is non-zero.
\end{prop}
To prove the above proposition, we first prove the following lemma.
\begin{lem}\label{lemma}
The series 
$$
\sum_{p~{\rm prime}}\frac{a_{\pi \times {\pi}}(p)}{p}
$$
is divergent.
\end{lem}
\begin{proof}
To prove the lemma, we prove that
\begin{equation}\label{prime}
\sum_{p\le x}a_{\pi \times {\pi}}(p)\sim \frac{x}{\log{x}},
\end{equation}
where $p$ varies over all primes less than or equal to $x$.
To prove \eqref{prime}, we first show that
\begin{equation}\label{prime-number}
\sum_{n\le x}a_{\pi \times {\pi}}(n)\Lambda(n)\sim x,
\end{equation}
where $\Lambda(n)$ is the von Mangoldt function.
Then by partial summation, we get \eqref{prime}. We apply Tauberian theorem to the 
Dirichlet series 
$$
-\frac{L'(s, \pi \times \pi)}{L(s, \pi \times \pi)}=
\sum_{n=1}^{\infty}\frac{a_{\pi \times {\pi}}(n)\Lambda(n)}{n^s}.
$$
By a result of Shahidi \cite[Theorem 5.2]{S}, $L(s, \pi \times \pi)$ does not vanish on $\Re(s)=1$. 
Also the series
$$
-\frac{L'(s, \pi \times \pi)}{L(s, \pi \times \pi)}=
\sum_{n=1}^{\infty}\frac{a_{\pi \times {\pi}}(n)\Lambda(n)}{n^s}.
$$
converges absolutely for $\Re(s)>1$ and it is analytically continued for $\Re(s)\ge 1$ 
except for a simple pole at $s=1$ since $L(s, \pi \times \pi)$ is analytically continued to the whole
complex plane except for a simple pole at $s=1$. Thus applying the Tauberian theorem, we get
\eqref{prime-number} and hence \eqref{prime}.

Now by partial summation, we have
$$
\sum_{p\le x}\frac{a_{\pi \times {\pi}}(p)}{p}=
\frac{A(x)}{x}+\int_{2}^{x}\frac{A(t)}{t^2}dt,
$$
where $A(x)=\sum_{p\le x}a_{\pi \times {\pi}}(p)$. Now using \eqref{prime}, we get
$$
\sum_{p\le x}\frac{a_{\pi \times {\pi}}(p)}{p}=
\frac{1}{\log{x}}+o\left(\frac{1}{\log{x}}\right)+
\int_{2}^{x}\left(\frac{1}{t\log{t}}+o\left(\frac{1}{t\log{t}}\right)\right)dt.
$$
By simplifying, we deduce that
$$
\sum_{p\le x}\frac{a_{\pi \times {\pi}}(p)}{p}=\log\log{x}+o(\log\log{x}).
$$
This proves the result.
\end{proof}

{\noindent{\bf Proof of \propref{proposition}:}}
Suppose that $\displaystyle\sum_{n=1}^{\infty}\frac{c(n)}{n}$ is zero. Since the series
$$
\sum_{n=1}^{\infty}\frac{a_{\pi \times {\pi}}(n)}{n^s}
$$
has a simple pole at $s=1$, the series
$$
\sum_{n=1}^{\infty}\frac{{|a_\pi(n)|}^2}{n^s}
$$
does not have any singularity at $s=1$. Then by Landau's theorem on Dirichlet series with 
non-negative coefficients, there exists a real number $\sigma$ such that the Dirichlet series
$$
\sum_{n=1}^{\infty}\frac{{|a_\pi(n)|}^2}{n^s}
$$
is convergent for  $\Re(s)>\sigma$ and it has a singularity at $s=\sigma$. Since
$g(s)$ is absolutely convergent for $\Re(s)>1/2$, we deduce that $\sigma\le 1/2$. But 
we will prove that the series
$$
\sum_{n=1}^{\infty}\frac{{|a_\pi(n)|}^2}{n}
$$
is divergent. This will lead to a contradiction. We know that for any prime $p$, we have
$$
{|a_{\pi}(p)|}^2=a_{\pi \times {\pi}}(p).
$$
Thus 
$$
\sum_{p~{\rm prime}}\frac{a_{\pi \times {\pi}}(p)}{p}\le \sum_{n=1}^{\infty}\frac{{|a_\pi(n)|}^2}{n},
$$
and the proposition follows from \lemref{lemma}.


\section{Proof of \thmref{improvement}}
We will apply \thmref{chandra-nara} to prove \thmref{improvement}. From \eqref{completed1} 
and \eqref{functional1}, we have
$$
\delta=1, ~~~q=0,~~~ r=0, ~~~\beta=1+\epsilon, ~~~A=\frac{m}{2}, ~~~
u=\frac{1}{2}-\frac{1}{2m}+\epsilon,~~~x'=x+O(x^{1-\eta-\frac{1}{m}}),
$$
for any $\eta>0$. Applying \thmref{chandra-nara}, we obtain
\begin{equation}\label{eq1}
\sum_{n\le x} a_{\pi}(n)\ll 
x^{\frac{1}{2}-\frac{1}{2m}+m\eta(\frac{1}{2}-\frac{1}{2m}+\epsilon)}
+\sum_{x\le n\le x'}|a_\pi(n)|.
\end{equation}
Now we estimate the sum $\sum_{x\le n\le x'}|a_\pi(n)|$. Since the coefficients $a_\pi(n)$
are multiplicative, it is easy to see that
\begin{equation}\label{eq2}
\sum_{x\le n\le x'}|a_\pi(n)|\ll \sum_{x\le n\le x' \atop{(n, N_\pi)=1}}|a_\pi(n)|.
\end{equation}
If $m=2$, we have 
$$
a_\pi(n)\ll n^{\frac{7}{64}+\epsilon}.
$$
Then for $m=2$, \eqref{eq1} becomes
$$
\sum_{n\le x} a_{\pi}(n)\ll x^{\frac{1}{4}+\frac{1}{2}\eta +2\epsilon} +
\sum_{x\le n\le x' \atop{(n, N_\pi)=1}}|a_\pi(n)|\ll x^{\frac{1}{4}+\frac{1}{2}\eta +2\epsilon} 
+x^{\frac{39}{64}-\eta+\epsilon}.
$$
Choosing $\eta=\frac{23}{96}$, we get
$$
\sum_{n\le x} a_{\pi}(n)\ll x^{\frac{71}{192}+\epsilon}.
$$
This proves the result for $m=2$. For $m\ge 3$,
using the Cauchy-Schwarz inequality and \eqref{ineq}, we have
\begin{equation}\label{eq3}
\sum_{x\le n\le x' \atop{(n, N_\pi)=1}}|a_\pi(n)|\ll 
x^{(1-\eta-\frac{1}{m})\frac{1}{2}}\left(\sum_{x\le n\le x' \atop{(n, N_\pi)=1}}|a_\pi(n)|^2\right)^{\frac{1}{2}}
\ll x^{\frac{1}{2}-\frac{\eta}{2}-\frac{1}{2m}}\left(\sum_{x\le n\le x' \atop{(n, N_\pi)=1}}
a_{\pi \times \tilde{\pi}}(n)\right)^{\frac{1}{2}}.
\end{equation}
By \propref{prop1}, we have 
\begin{equation}\label{eq4}
\sum_{x\le n\le x' \atop{(n, N_\pi)=1}}a_{\pi \times \tilde{\pi}}(n)
\ll x^{1-\eta-\frac{1}{m}}+ x^{\frac{m^2-1}{m^2+1}+\epsilon},
\end{equation}
for every $\eta\ge 0$ and $\epsilon>0$. We will choose $\eta$ in such a way that
$1-\eta-\frac{1}{m}\le \frac{m^2-1}{m^2+1}$. This will imply that the second term in the 
right hand side of the above estimate will be the dominant term. Now from \eqref{eq1},
\eqref{eq3} and \eqref{eq4}, we have
$$
\sum_{n\le x} a_{\pi}(n)\ll x^{\frac{1}{2}-\frac{1}{2m}+m\eta(\frac{1}{2}-\frac{1}{2m}+\epsilon)}+
x^{1-\eta-\frac{1}{m}}+x^{\frac{1}{2}-\frac{\eta}{2}-\frac{1}{2m}+\frac{m^2-1}{2(m^2+1)}+\epsilon}.
$$
Choosing $\eta=\frac{m^2-1}{m(m^2+1)}$ in the above estimate, we see that 
the exponent of the first term of the right hand side equals the exponent of third term 
and it is bigger than the exponent of the second term. Thus we obtain
$$
\sum_{n\le x} a_{\pi}(n)\ll x^{\frac{m^2-m}{m^2+1}+\epsilon},
$$
for any fixed $\epsilon>0$.
This completes the proof.

\section{Proof of \thmref{main}}
We first do the analysis for general $m\ge 2$ and then we will specialise to the case when $m=2$.
The bounds for the coefficients $a_\pi(n)$ of the automorphic $L$-function associated to an 
automorphic irreducible cuspidal representation $\pi$ of $GL_m$ over $\mathbb{Q}$ are given by
\begin{equation}\label{bound}
a_\pi(n)\ll 
\left\{\begin{array}{rcl}
n^{\frac{7}{64}+\epsilon} &\mbox{if} & m=2\\
n^{\frac{5}{14}+\epsilon} &\mbox{if} & m=3\\
n^{\frac{1}{2}-\frac{1}{m^2+1}+\epsilon} & \mbox{if} & m\ge 4,
\end{array}\right.
\end{equation}
for any $\epsilon >0$.

Thus in the notation of \thmref{thm1}, we have
$$
\alpha=
\left\{\begin{array}{rcl}
\frac{7}{64}+\epsilon &\mbox{if} & m=2,\\
\frac{5}{14}+\epsilon &\mbox{if} & m=3\\
\frac{1}{2}-\frac{1}{m^2+1}+\epsilon & \mbox{if} & m\ge 4,
\end{array}\right.
$$
from \thmref{improvement}, we have
$$
\beta=
\left\{\begin{array}{rcl}
\frac{71}{192}+\epsilon & \mbox{if} & m=2,\\
\frac{m^2-m}{m^2+1}+\epsilon & \mbox{if} & m\ge 3,
\end{array}\right.
$$
Now we need to find out $\gamma$. Since $\pi$ is a self dual representation, if $\alpha<1/4$, 
then by \thmref{gun-murty}, we have
\begin{equation}\label{gm}
\sum_{n=1}^{\infty}\frac{a_\pi(n)^2}{n^s}=
g(s)\sum_{n=1}^{\infty}\frac{a_{\pi \times {\pi}}(n)}{n^s},
\end{equation}
where $g(s)$ is absolutely convergent for $\Re(s)>1/2$. Let $g(s)$ be given by the Dirichlet series
$$
g(s)=\sum_{n=1}^{\infty}\frac{c(n)}{n^s}.
$$
Then from \eqref{gm}, for each $n\ge 1$, we have
\begin{equation}
a_\pi(n)^2=\sum_{d|n}c(d)a_{\pi \times {\pi}}(n/d).
\end{equation}
Using \propref{prop1} and the above equation, we get
$$
\sum_{n\le x}a_\pi(n)^2=
\sum_{d\le x}c(d)\left \{c_\pi\frac{x}{d}+
O\left(\left(\frac{x}{d}\right)^{\frac{m^2-1}{m^2+1}+\epsilon}\right)\right\}.
$$
Thus

\begin{equation}\label{main2}
\sum_{n\le x}a_\pi(n)^2=c_\pi x \sum_{d\le x}\frac{c(d)}{d}
+O(x^{\frac{m^2-1}{m^2+1}+\epsilon}).
\end{equation}
The series 
$$
g(s)=\sum_{n=1}^{\infty}\frac{c(n)}{n^s}
$$
is absolutely convergent for $\Re(s)>1/2$, and by partial summation, for any $\theta>1/2$, we have
$$
\sum_{d>x}\frac{c(d)}{d^\theta}\ll \int_{x}^{\infty}\frac{B(t)}{t^{1+\theta}}dt,
$$
where  
$$
B(t)=\sum_{n\le t} c(n)\ll t^{\frac{1}{2}+\epsilon}
$$
for any $\epsilon>0$. By simplifying, we get
\begin{equation}\label{main3}
\sum_{d>x}\frac{c(d)}{d^\theta}\ll x^{-\theta+\frac{1}{2}+\epsilon}.
\end{equation}
Inserting this estimate in \eqref{main2}, we obtain
$$
\sum_{n\le x}a_\pi(n)^2=c_\pi g(1)x+O(x^{\frac{1}{2}+\epsilon})+O(x^{\frac{m^2-1}{m^2+1}+\epsilon}).
$$
Since $g(1)\ne 0$ by \propref{proposition}, we deduce that
$$
\sum_{n\le x}a_\pi(n)^2=c_1x+O(x^{\frac{1}{2}+\epsilon})+O(x^{\frac{m^2-1}{m^2+1}+\epsilon}),
$$
where $c_1$ is a positive constant. 
Since $m\ge 2$, it again simplifies to
\begin{equation}
\sum_{n\le x}a_\pi(n)^2=c_1x+O(x^{\frac{m^2-1}{m^2+1}+\epsilon})
\end{equation}
for some positive constant $c_1$.

Thus we have found that if $\alpha<1/4$, then $\gamma=\frac{m^2-1}{m^2+1}+\epsilon$ for any 
$\epsilon>0$. For $m=2$, $\alpha=7/64 +\epsilon<1/4$, $\beta=\frac{71}{192}+\epsilon$ and
$\gamma=\frac{3}{5}+\epsilon$. Therefore for $m=2$, we have 
$$
\mbox{max}\{\alpha+\beta, \gamma\}=\frac{3}{5}+\epsilon<1,
$$
Thus by \thmref{thm1}, \thmref{main} follows.

\section{Proof of \thmref{siegel}}
The idea of proof is to apply \thmref{thm1} to the sequence $\{\lambda(n)\}_{n\ge 1}$.
By \eqref{siegel-ramanujan} we have $\alpha=\epsilon$ for any $\epsilon>0$ in this case.
It has been proved in \cite[Lemma 1]{RSW} that
$$
\sum_{n\le x}a_F(n)\ll x^{\frac{3}{5}+\epsilon}.
$$
Using the above estimate and \eqref{mobius} we deduce that
$$
\sum_{n\le x} \lambda(n)\ll x^{\frac{3}{5}+\epsilon}.
$$ 
Thus $\beta=\frac{3}{5}+\epsilon$ in this case. Also from \cite[Theorem 1.1]{JL}, we have the following estimate for the second moment of $\lambda(n)$. For any $\epsilon>0$, we have
$$
\sum_{n\le x}\lambda(n)^2=c_Fx + 
O_\epsilon\left(x^{\frac{41}{47}+\epsilon}+k^{\frac{30}{17}}x^{\frac{11}{17}+\epsilon}\right),
$$
where $c_F$ is a positive constant depending on $F$. Thus 
$\gamma=\frac{41}{47}+\epsilon$ in this case. Therefore
$$
\rm{max}\{\alpha+\beta, \gamma \}=\frac{41}{47}+\epsilon<1.
$$
Hence by \thmref{thm1}, \thmref{siegel} follows.
\section{Further remarks}
The Ramanujan conjecture is known for certain automorphic forms for 
$GL_m(\mathbb{A}_{\mathbb{Q}})$.
If we assume the Ramanujan conjecture on the coefficients $a_\pi(n)$ attached to an automorphic form, then we also get a result for the number of sign changes in the interval $[x,2x]$. Here also we note that the bounds obtained by Goldfeld and Sengupta \cite{GS} are sufficient to prove the following theorem.
\begin{thm}\label{remar}
Let $L(s,\pi)$ be the automorphic $L$-function associated to an automorphic irreducible 
self-dual cuspidal representation $\pi$ of $GL_m(\mathbb{A}_{\mathbb{Q}})$ with unitary central
character satisfying the Ramanujan conjecture, and let $a_\pi(n)$ be its $n^{th}$ coefficient. If the
sequence $\{a_\pi(n)\}_{n=1}^{\infty}$ is a sequence of real numbers, then for any real number $r$
satisfying $\frac{m^2-1}{m^2+1}<r<1$, the sequence $\{a_\pi(n)\}_{n=1}^{\infty}$ has at least one sign change for $n\in [x,x+x^r]$. Consequently, the number of sign changes of $a_\pi(n)$ for $n\le x$ is 
$\gg x^{1-r}$ for sufficiently large $x$.
\end{thm}
The Ramanujan conjecture is known to be true for the coefficients of symmetric power $L$-functions attached to holomorphic Hecke cusp forms. Also the automorphicity of symmetric power $L$-functions 
$Sym^j(s,f)$ attached to a holomorphic Hecke cusp form $f$ is known for some values of $j$. Therefore in these cases we can apply \thmref{remar}. Also the exponent towards the Ramanujan conjecture attached to 
the symmetric square $L$-function attached to any Hecke Maass cusp form is $\frac{14}{64}<\frac{1}{4}$.
Therefore in this case also we get a result similar to \thmref{main} on sign changes for the coefficients.

\par
\par
\noindent{\bf Acknowledgements.} We thank the anonymous referee for giving valuable suggestions.
 Research of the second author was supported by an NSERC Discovery grant.

\end{document}